\documentclass[amssymb,twoside,12pt]{article}
\usepackage{hyperref}
\usepackage{latexsym,amsmath}
\usepackage{cleveref}
\usepackage{amsfonts}
\usepackage{graphicx}
\usepackage{epstopdf}
\usepackage{float}
\numberwithin{equation}{section}
\newtheorem{theorem}{Theorem}[section]
\newtheorem{lemma}[theorem]{{\bf Lemma}}
\newtheorem{proposition}[theorem]{{\bf Proposition}}

\newenvironment{proof}{\noindent{\em Proof:}}{\quad \hfill
$\Box$\vspace{2ex}}

 \newtheorem{remark}{Remark}[section]

\begin{document}
\setcounter{page}{1}
\begin{center}

\vspace{0.4cm} {\Large\bf  $l_1$-$l_2$ Regularization of Split Feasibility Problems}
\vspace{0.4cm}

 Abdellatif {\sc Moudafi} \footnote{Aix Marseille Universit\'e,  CNRS-L.I.S UMR 7296, Domaine Universitaire de Saint-J\'er\^ome. Avenue Escadrille Normandie-Niemen, 13397 Marseille, abdellatif.moudafi@univ-amu.fr} and Aviv {\sc Gibali}
\footnote{(Corresponding author) ORT Braude College,  Department of Mathematics, Karmiel 2161002, Israel, avivg@braude.ac.il}
\\

\end{center}
\def\baselinestretch{1.0}\small

\begin{abstract}

Numerous problems in signal processing and imaging, statistical learning and data mining, or computer
vision  can be formulated as optimization problems which consist in minimizing a sum of convex functions,
not necessarily differentiable, possibly composed with linear operators and that in turn can be transformed  to split feasibility problems (SFP), see for example \cite{ce94}. Each function is typically either a data fidelity term or a regularization term
enforcing some properties on the solution, see for example \cite{cpp09} and references therein. In this paper we are interested in  Split Feasibility Problems which can be seen as a general form of  $Q$-Lasso  introduced in \cite{aasnx13} that extended the well-known Lasso of Tibshirani \cite{Tibshirani96}.
$Q$ is a closed convex subset of a Euclidean $m$-space, for some integer $m\geq1$, that can be interpreted as the set of errors  within given tolerance level when linear measurements are taken to recover a signal/image via the Lasso.
 Inspired by  recent works by Lou et al \cite{ly, xcxz12},  we  are interested in a nonconvex regularization of SFP and  propose  three split algorithms  for  solving  this general case. The first one is based on   the DC (difference of convex) algorithm (DCA)  introduced by Pham Dinh Tao, the second one in nothing else than the celebrate forward-backward algorithm and the third one uses a method introduced by Mine and Fukushima. It is worth mentioning that the SFP model
 a number of applied problems arising from  signal/image processing and specially optimization  problems for intensity-modulated radiation therapy (IMRT) treatment planning, see for example \cite{cbmt06}.\\

 \smallskip

 \smallskip
\end{abstract}

\section{\protect\small Introduction and Preliminaries}

{\small Recent developments in science and technology have caused a
revolution in data processing, as large datasets are becoming increasingly
available and important. To meet the need in big data area, the field of
compressive sensing (CS) \cite{Donoho16} is rapidly blooming. The process of
CS consists of encoding and decoding. The process of encoding involves
taking a set of (linear) measurements, $b=Ax$, where $A$ is a matrix of size
$m\times n$. If $m<n$, we say the signal $x\in I\!\!R^{n}$ can be
compressed. The process of decoding is to recover $x$ from $b$ with an
additional assumption that $x$ is sparse. It can be expressed as an
optimization problem,
\begin{equation}
\min \Vert x\Vert _{0}\ \ \mbox{subject to}\ Ax=b,
\end{equation}%
with $\Vert \cdot \Vert _{0}$ being the $l_{0}$ norm, which counts the
number of nonzero entries of $x$; that is
\begin{equation}
\Vert x\Vert _{0}=|\{x_{i}\mid x_{i}\neq 0\}|
\end{equation}%
where $|\cdot |$ denotes here the cardinality, i.e., the number of elements
of a set. So minimizing the $l_{0}$ norm is equivalent to finding the
sparsest solution. One of the biggest obstacles in CS is solving the
decoding problem above, as $l_{0}$ minimization is NP-hard. A popular
approach is to replace $l_{0}$ by the convex norm $l_{1}$, which often gives
a satisfactory sparse solution. This $l_{1}$ heuristic has been applied in
many different fields such as geology and geophysics, spectroscopy, and
ultrasound imaging.\newline
Recently, there has been an increase in applying nonconvex metrics as
alternative approaches to $l_{1}$ . In particular, the nonconvex metric $%
l_{p}$ for $p\in (0,1)$ in \cite{Chartrand07} can be regarded as a
continuation strategy to approximate $l_{0}$ as $p\rightarrow 0$. The
optimization strategies include iterative reweighting \cite{Chartrand07} and
half thresholding \cite{xcxz12}, the scale-invariant $l_{1}$, formulated as
the ratio of $l_{1}$ and $l_{2}$, was discussed in \cite{elx13}. Other
nonconvex $l_{1}$ variants include transformed $l_{1}$, sorted $l_{1}$ and
capped $l_{1}$. It is demonstrated in a series of papers \cite{ly, xcxz12}
that difference of the $l_{1}$ and $l_{2}$ norms, denoted as $l_{1}$-$l_{2}$%
, outperforms $l_{1}$ and $l_{p}$ in terms of promoting sparsity when
sensing matrix $A$ is highly coherent. Based on this idea, we propose the
same type of regularization for SFP and propose three splitting algorithms,
the first one is nothing but the DC (difference of convex) algorithm (DCA)
introduced by Pham Dinh Tao see for example \cite{Natarajan95}, the second
one in nothing else than the celebrate forward-backward algorithm and the
third one uses a method introduced by Mine and Fukushima in \cite{mf81} for
minimizing the sum of a convex function and a differentiable one.\smallskip
\newline
First, remember that the lasso of Tibshirani \cite{Tibshirani96} is the
minimization problem
\begin{equation}
\min_{x\in I\!\!R^{n}}\frac{1}{2}\Vert Ax-b\Vert _{2}^{2}+\gamma \Vert
x\Vert _{1},
\end{equation}%
where $A$ is an $m\times n$ real matrix, $b\in I\!\!R^{m}$ and $\gamma >0$
is a tuning parameter. It is equivalent to the the basic pursuit (BP) of
Chen et al. \cite{cds98}
\begin{equation}
\min_{x\in I\!\!R^{n}}\Vert x\Vert _{1}\quad \mbox{subject to}\ Ax=b.
\label{Eq:1.2}
\end{equation}%
However, due to errors of measurements, the constraint $Ax=b$ is actually
inexact; It turns out that problem (\ref{Eq:1.2}) is reformulated as
\begin{equation}
\min_{x\in I\!\!R^{n}}\Vert x\Vert _{1}\quad \mbox{subject to}\ \Vert
Ax-b\Vert _{p}\leq \varepsilon ,  \label{Eq:1.3}
\end{equation}%
where $\varepsilon >0$ is the tolerance level of errors and $p$ is often $%
1,2 $ or $\infty $. It is noticed in \cite{aasnx13} that if we let $%
Q:=B_{\varepsilon }(b)$, the closed ball in $I\!\!R^{n}$ with center $b$ and
radius $\varepsilon $, then (\ref{Eq:1.3}) is rewritten as
\begin{equation}
\min_{x\in I\!\!R^{n}}\Vert x\Vert _{1}\quad \mbox{subject to}\ Ax\in Q.
\end{equation}%
With $Q$ a nonempty closed convex set of $I\!\!R^{m}$ and $P_{Q}$ the
projection from $I\!\!R^{m}$ onto $Q$ and since that the constraint is
equivalent to the condition $Ax-P_{Q}(Ax)=0$, this leads to the following
equivalent Lagrangian formulation
\begin{equation}
\min_{x\in I\!\!R^{n}}\frac{1}{2}\Vert (I-P_{Q})Ax\Vert _{2}^{2}+\gamma
\Vert x\Vert _{1},  \label{Eq:1.5}
\end{equation}%
with $\gamma >0$ a Lagrangian multiplier. A connection is also made in \cite%
{aasnx13} with the so-called split feasibility problem \cite{ce94} which is
stated as finding $x$ verifying
\begin{equation}
x\in C,\quad Ax\in Q,  \label{Eq:1.6}
\end{equation}%
where $C$ and $Q$ are closed convex subsets of $I\!\!R^{n}$ and $I\!\!R^{m}$%
, respectively. An equivalent minimization formulation of (\ref{Eq:1.6}) is
\begin{equation}
\min_{x\in C}\frac{1}{2}\Vert (I-P_{Q})Ax\Vert _{2}^{2}.
\end{equation}%
Its $l_{1}$ regularization is given as
\begin{equation}
\min_{x\in C}\frac{1}{2}\Vert (I-P_{Q})Ax\Vert _{2}^{2}+\gamma \Vert x\Vert
_{1},
\end{equation}%
where $\gamma >0$ is a regularization parameter.\newline
This convex relaxation attracts considerable attention see for example \cite%
{aasnx13} and references there in. In this paper we study a non-convex but
Lipschitz continuous metric $l_{1}$-$l_{2}$ for SFP. As illustrated in \cite%
{ly} the level curves of $l_{1}$-$l_{2}$ are closer to $l_{0}$ than those of
$l_{1}$, which motivated us to consider the nonconvex $l_{1}$-$l_{2}$
regularization for split feasibility problem, namely
\begin{equation}
\min_{x\in C}\frac{1}{2}\Vert (I-P_{Q})Ax\Vert _{2}^{2}+\gamma (\Vert x\Vert
_{1}-\Vert x\Vert _{2}),  \label{Eq:1.9}
\end{equation}%
and propose three algorithms. The first uses the DCA which is a descent
method without line search introduced by Tao and An \cite{Natarajan95} for
minimizing a function $f$ which is the difference of two lower
semicontinuous proper convex functions $g$ and $h$ on the space $I\!\!R^{n}$%
. The second one is based on the gradient proximal method to solve the
problem (\ref{Eq:1.9}) by full splitting, that is, at every iteration, the
only operations involved are evaluations of the gradient of the function $%
\frac{1}{2}\Vert (I-P_{Q})A(\cdot )\Vert _{2}^{2}$, the proximal mapping of $%
\Vert \cdot \Vert _{1}-\Vert \cdot \Vert _{2}$, $A$, or its transpose $A^{t}$%
. The third one is based on an algorithm for minimizing the sum of a convex
function and a differentiable one introduced by Mine and Fukushima in \cite%
{mf81}.\smallskip \newline
In \cite{aasnx13}, properties and iterative methods for (\ref{Eq:1.5}) are
investigated. Remember also that many authors devoted their works to the
unconstrained minimization problem $\min_{x\in H}f_{1}(x)+f_{2}(x)$ with $%
f_{1},f_{2}$ are two proper, convex lower semi continuous functions defined
on a Hilbert space $H$ and $f_{2}$ differentiable with a $\beta $-Lipschitz
continuous gradient for some $\beta >0$ and an effective method to solve it
is the forward-backward algorithm which from an initial value $x_{0}$
generates a sequence $(x_{k})$ by the following iteration
\begin{equation}
x_{k+1}=(1-\lambda _{k})x_{k}+\lambda _{k}prox_{\gamma
_{k}f_{1}}(x_{k}-\gamma _{k}\nabla f_{2}(x_{k})),
\end{equation}%
where $\gamma _{k}>0$ is the algorithm step-size, $0<\lambda _{k}<1$ is a
relaxation parameter and $prox_{\gamma _{k}f_{1}}$ being the proximal
mapping defined in (\ref{eq:Prox_def}).\newline
It is well-known, see for instance \cite{cp07}, that if $(\gamma _{k})$ is
bounded and $(\lambda _{k})$ is bounded from below, then $(x_{k})$ weakly
converges to a solution of $\min_{x\in H}f_{1}(x)+f_{2}(x)$ provided that
the set of solutions is nonempty.\newline
In order to relax the assumption on the differentiability of $f_{2}$, the
Douglas-Rachford algorithm was introduced. It generates a sequence $(y_{k})$
as follows
\begin{equation}
\left\{
\begin{array}{l}
\ y_{k+1/2}=prox_{\kappa f_{2}}y_{k}; \\
\ y_{k+1}=y_{k}+\ \tau _{k}\big(prox_{\kappa
f_{1}}(2y_{k+1/2}-y_{k})-y_{k+1/2}\big)%
\end{array}%
\right.
\end{equation}%
where $\kappa >0$, $(\tau _{k})$ is a sequence of positive reals. It is
well-known that $(y_{k})$ converges weakly to $y$ such that $prox_{\kappa
f_{2}}y$ is a solution of the unconstrained minimization problem above
provided that: $\forall k\in I\!\!N,\ \tau _{k}\in ]0,2[$ and $%
\sum_{k=0}^{\infty }\tau _{k}(2-\tau _{k})=+\infty $ and the set of
solutions is nonempty.\smallskip \newline
In what follows we are interested in (\ref{Eq:1.9}) which is more
challenging and we will focus our attention on the algorithmic aspect.  }

{\small Our paper is organized as follows. In Section \ref{Sec:Comp}, we
first start with definitions and notions which are needed for the
presentation of our three proposed schemes, the DCA algorithm, the
forward-backward algorithm and the third based on Mine-Fukushima algorithm.
We also give full convergence theorem for the proposed schemes. Later in
Section \ref{Sec:Num}, we present several numerical experiments which
illustrates the performances of our schemes compared with the CQ and relaxed
CQ algorithms. We include random linear system of equations as well as an
example in sparse signal recovery. Finally, in Section \ref{Sec:App} we
provide further insights into how to compute the proximal mapping of a sum
of two functions by coupling the Douglas-Rachford and the forward-backward
algorithms.  }

\section{\protect\small Computational approaches}

{\small \label{Sec:Comp}  }

\subsection{\protect\small DCA}

{\small First, remember that the \textit{subdifferential set} (or just subdifferential) of a convex
function $h$ is defined as
\begin{equation}
\partial h(x):=\{u\in I\!\!R^{n};h(y)\geq h(x)+\langle u,y-x\rangle \
\forall y\in I\!\!R^{n}\}.
\end{equation}%
Each element of $\partial h(x)$ is called \textit{subgradient}. In case that the function $h$ is continuously differentiable then $\partial h(x) = \{\nabla h(x)\}$, this is the \textit{gradient} of $h$. It is easily seen that
\begin{equation}
\partial \frac{1}{2}\Vert Ax-y\Vert ^{2}=\nabla \frac{1}{2}\Vert Ax-y\Vert
^{2}=A^{t}(Ax-y),  \label{Eq:grad_norm2}
\end{equation}%
and
\begin{equation}\label{Eq:Subdiff}
(\partial \Vert x\Vert _{1})_{i}=\left\{
\begin{array}{ll}
\ sgn(x_{i}) & \mbox{if}\ \ x_{i}\not=0; \\
\ \mbox{any element of }\lbrack -1,1] & \mbox{if}\ \ x_{i}=0. \\
&
\end{array}%
\right.
\end{equation}
The \textit{characteristic function} of a set $C\subseteq I\!\!R^{n}$ is
defined as
\begin{equation}
i_{C}(x)=\left\{
\begin{array}{ll}
\ 0 & \mbox{if}\ \ x\in C; \\
\ +\infty & {\mbox otherwise}\ \
\end{array}%
\right.
\end{equation}%
such function is convenient to enforce hard constraints on the solution.
Moreover, the \textit{normal cone} of $C$ at $x\in C$, denoted by $%
N_{C}\left( x\right) $ is defined
\begin{equation}
N_{C}\left( x\right) :=\{d\in I\!\!R^{n}\mid \left\langle d,y-x\right\rangle
\leq 0,\text{ }\forall y\in C\}.  \label{eq:normal_cone}
\end{equation}%
A known relation between the above definition is that $\partial i_{C}=N_{C}$%
. Another useful definition which will be useful in the sequel is the
following. A sequence $(x_{k})$ is called \textit{asymptotically regular},
if $\lim_{n\rightarrow \infty }\Vert x_{k+1}-x_{k}\Vert =0$.\newline
}

{\small For finding critical points of $f:=g-h$, the DCA involves the
construction of two sequences $(x_k)$ and $(y_k)$ by the following rules  }

{\small
\begin{equation}
\left\{
\begin{array}{l}
\ y_{k}\in \partial h(x_{k}); \\
\ x_{k+1}=argmin_{x\in I\!\!R^{n}}\big(g(x)-(h(x_{k})+\langle
y_{k},x-x_{k}\rangle )\big).%
\end{array}%
\right.  \label{eq:2.6}
\end{equation}%
Note that by the definition of subdifferential , we can write
\begin{equation}
h(x_{k+1})\geq h(x_{k})+\langle y_{k},x_{k+1}-x_{k}\rangle .
\end{equation}%
Since $x_{k+1}$ minimizes $g(x)-(h(x_{k})+\langle y_{k},x-x_{k}\rangle )$,
we also have
\begin{equation}
g(x_{k+1})-(h(x_{k})+\langle y_{k},x_{k+1}-x_{k}\rangle )\leq
g(x_{k})-h(x_{k}).
\end{equation}%
Combining the last inequalities, we obtain
\begin{equation}
f(x_{k})=g(x_{k})-h(x_{k})\geq g(x_{k+1})-(h(x_{k})+\langle
y_{k},x_{k+1}-x_{k}\rangle )\geq f(x_{k+1}).
\end{equation}%
Therefore, the DCA provides a monotonically decreasing sequence $(f(x_{k}))$
which converges provided that the objective function $f$ is bounded below.%
\newline
The objective function in (\ref{Eq:1.9}) has the following DC decomposition
\begin{equation}
\min_{x\in I\!\!R^{n}}\left( \frac{1}{2}\Vert (I-P_{Q})Ax\Vert
_{2}^{2}+\gamma \Vert x\Vert _{1}+i_{C}(x)\right) -\gamma \Vert x\Vert _{2}.
\end{equation}%
Observe that $\Vert x\Vert _{2}$ is differentiable with gradient $x/\Vert
x\Vert _{2}$ for any $x\not=0$ and we also have $0\in \partial \Vert \cdot
\Vert _{2}(0)$, which leads to the following iterates
\begin{equation}
x_{k+1}=\left\{
\begin{array}{ll}
\ argmin_{x\in I\!\!R^{n}}\frac{1}{2}\Vert (I-P_{Q})Ax\Vert _{2}^{2}+\gamma
\Vert x\Vert _{1}+i_{C}(x) & \mbox{if}\ x_{k}=0 \\
\ argmin_{x\in I\!\!R^{n}}\frac{1}{2}\Vert (I-P_{Q})Ax\Vert _{2}^{2}+\gamma
\Vert x\Vert _{1}+i_{C}(x)-\left\langle x,\gamma \frac{x_{k}}{\Vert
x_{k}\Vert _{2}}\right\rangle & \mbox{if}\ x_{k}\not=0,%
\end{array}%
\right.  \label{Alg:2.3}
\end{equation}%
obtained by setting in the rules (\ref{eq:2.6}): $g(x)=1/2\Vert
(I-P_{Q})Ax\Vert _{2}^{2}+\gamma \Vert x\Vert _{1}+i_{C}(x)$ and $%
h(x)=\gamma \Vert x\Vert _{2}$. (\ref{Alg:2.3}) is equivalent, using the
definition of the characteristic function, to
\begin{equation}
x_{k+1}=\left\{
\begin{array}{ll}
\ argmin_{x\in C}\frac{1}{2}\Vert (I-P_{Q})Ax\Vert _{2}^{2}+\gamma \Vert
x\Vert _{1} & \mbox{if}\ x_{k}=0 \\
\ argmin_{x\in C}\frac{1}{2}\Vert (I-P_{Q})Ax\Vert _{2}^{2}+\gamma \Vert
x\Vert _{1}-\langle x,\gamma \frac{x_{k}}{\Vert x_{k}\Vert _{2}}\rangle & %
\mbox{if}\ x_{k}\not=0.%
\end{array}%
\right.
\end{equation}%
\ Now, we define for all $\gamma >0$, the following function
\begin{equation}
\Gamma (x)=\frac{1}{2}\Vert (I-P_{Q})Ax\Vert _{2}^{2}+\gamma (\Vert x\Vert
_{1}-\Vert x\Vert _{2})+i_{C}(x).
\end{equation}%
We are in a position to prove the following convergence properties of the
iterative step (\ref{Alg:2.3}):  }

\begin{proposition}
{\small Let $(x_k)$ be the sequence generated by Algorithm \ref{Alg:2.3}.%
\newline
(i) For all $\gamma>0$ we have that $\lim_{\Vert x\Vert_2\rightarrow
+\infty}\Gamma(x)=+\infty$. $\Gamma$ is therefore coercive in the sense that
its levels sets are bounded, namely $\{x\in I\!\!R^n; \Gamma(x)\leq
\Gamma(x_0)\}$ is bounded for any $x_0\in I\!\!R^n$.\newline
(ii) The sequence $(x_k)$ is bounded.\newline
(iii) If $\lim_{k\rightarrow +\infty}\Vert x_{k+1}-x_k\Vert_2=0$, i.e. $%
(x_k) $ is asymptotically regular, then any nonzero limit point $x^*$ of the
sequence $(x_k)$ is a stationary point of (\ref{Eq:1.9}), namely
\begin{equation}
0\in A^t(I-P_Q)Ax^*+\gamma\left(\partial \Vert x^*\Vert_1-\frac{x^*}{\Vert
x^*\Vert_2}\right)+N_C(x^*).
\end{equation}
}
\end{proposition}

\begin{proof}
{\small \ Recall first that the support of $x$ is defined by $\mbox{supp}%
(x)=\{1\leq i\leq n;x_{i}\not=0\}$ and that $\Vert x\Vert _{0}=|\mbox{supp}%
(x)|$ is the cardinality of $\mbox{supp}(x)$. To prove (i)-(ii), remember
that for all $x\not=0$, we have $\Vert x\Vert _{1}-\Vert x\Vert _{2}\geq 0$
and that $\Vert x\Vert _{1}-\Vert x\Vert _{2}=0\Leftrightarrow \Vert x\Vert
_{0}=1$. With this fact in hand, we can easily verify that $\Gamma $ is
coercive. \newline
Now, a simple computation which uses the fact that $\Vert a\Vert ^{2}-\Vert
b\Vert ^{2}=\Vert a-b\Vert ^{2}+2\langle b,a-b\rangle $, gives
\begin{eqnarray}
\Gamma (x_{k})-\Gamma (x_{k+1}) &=&\frac{1}{2}\Vert Ax_{k}-Ax_{k+1}-\big(%
P_{Q}(Ax_{k})-P_{Q}(Ax_{k+1})\big)\Vert ^{2}  \notag \\
&+&\langle Ax_{k}-Ax_{k+1}-\big(P_{Q}(Ax_{k})-P_{Q}(Ax_{k+1})\big)%
,Ax_{k+1}-P_{Q}(Ax_{k+1})\rangle  \notag \\
&+&\gamma (\Vert x_{k}\Vert _{1}-\Vert x_{k+1}\Vert _{1}-\Vert x_{k}\Vert
_{2}+\Vert x_{k+1}\Vert _{2}).  \label{eq:star}
\end{eqnarray}%
The first-order optimality condition at $x_{k+1}$ as the solution of the
problem (\ref{Alg:2.3}) and the fact that $\partial (\Vert \cdot \Vert
_{1}+i_{C})(x)=\partial \Vert x\Vert _{1}+N_{C}(x)$ (since a norm is
continuous) lead to.
\begin{equation*}
A^{t}(I-P_{Q})Ax_{k+1}+\gamma (w_{k+1}-y_{k})+p_{k+1}=0,
\end{equation*}%
with $y_{k}\in \partial \Vert x_{k}\Vert _{2}$, $w_{k+1}\in \partial \Vert
x_{k+1}\Vert _{1}$ and $p_{k+1}\in N_{C}(x_{k+1})$. This combined with $%
\langle w_{k},x_{k+1}\rangle =\Vert x_{k+1}\Vert _{1}$ gives
\begin{equation}
\langle A(x_{k}-x_{k+1}),(I-P_{Q})Ax_{k+1}\rangle +\gamma (\langle
w_{k+1},x_{k}\rangle -\Vert x_{k+1}\Vert _{1}+\langle
y_{k},x_{k+1}-x_{k}\rangle )-\langle p_{k+1},x_{k+1}-x_{k}\rangle =0.
\label{Eq:2.6}
\end{equation}
}

{\small Combining (\ref{eq:star}) and (\ref{Eq:2.6}), we can write
\begin{eqnarray}
\Gamma (x_k)-\Gamma (x_{k+1})&=&\frac{1}{2}\Vert Ax_k-Ax_{k+1}-\big(%
P_Q(Ax_k)-P_Q(Ax_{k+1})\big)\Vert^2  \notag \\
&-& \gamma\big( \langle w_{k+1},x_k\rangle -\Vert x_{k+1}\Vert_1+ \langle
y_k, x_{k+1}-x_k\rangle\big)+\langle p_{k+1}, x_{k+1}-x_k\rangle  \notag \\
&-& \langle Ax_{k+1}-P_Q(Ax_{k+1}), P_Q(Ax_k)-P_Q(Ax_{k+1})\rangle  \notag \\
&+& \gamma ( \Vert x_{k}\Vert_1-\Vert x_{k+1}\Vert_1-\Vert
x_{k}\Vert_2+\Vert x_{k+1}\Vert_2)\rangle.  \label{Eq:stam}
\end{eqnarray}
}

{\small The characterization of the orthogonal projection, namely
\begin{equation}
\langle x-P_Q(x),z-P_Q(x)\rangle \leq 0 \ \forall z\in Q,
\end{equation}
assures that
\begin{equation}
\langle (I-P_Q)Ax_{k+1}, P_Q(Ax_k)-P_Q(Ax_{k+1}) \rangle\leq 0,
\end{equation}
and thus
\begin{eqnarray}
\Gamma (x_{k})-\Gamma (x_{k+1}) &\geq &\frac{1}{2}\Vert
(I-P_{Q})(Ax_{k})(I-P_{Q})(Ax_{k+1})\Vert ^{2}+\gamma (\Vert x_{k}\Vert
_{1}-\langle w_{k+1},x_{k}\rangle )  \notag \\
&+&\gamma (\Vert x_{k+1}\Vert _{2}-\Vert x_{k}\Vert _{2}-\langle
y_{k},x_{k+1}-x_{k}\rangle )+\langle p_{k+1},x_{k+1}-x_{k}\rangle .  \notag
\label{Eq:stam1}
\end{eqnarray}
}

{\small On the other hand, since $|w_{k+1,i}|\leq 1$ for $i=1,...,n$, $y_{k}\in
\partial \Vert x_{k}\Vert _{2}$ and $p_{k+1}\in N_{C}(x_{k+1})$, we also
have
\begin{equation}
\Vert x_{k}\Vert _{1}-\langle w_{k+1},x_{k}\rangle \geq 0\ \Vert
x_{k+1}\Vert _{2}-\Vert x_{k}\Vert _{2}-\langle y_{k},x_{k+1}-x_{k}\rangle
\geq 0\ \mbox{and}\ \langle p_{k+1},x_{k+1}-x_{k}\rangle \geq 0.
\end{equation}%
Consequently,
\begin{eqnarray}
\Gamma (x_{k})-\Gamma (x_{k+1}) &\geq &\frac{1}{2}\Vert
(I-P_{Q})(Ax_{k})-(I-P_{Q})(Ax_{k+1})\Vert ^{2}  \notag \\
+\gamma (\Vert x_{k+1}\Vert _{2}-\Vert x_{k}\Vert _{2}-\langle
y_{k},x_{k+1}-x_{k}\rangle ) &\geq &0.  \label{Eq:2.7}
\end{eqnarray}%
This ensures that the sequence $(\Gamma (x_{k}))$ is monotonically
decreasing, which in turn ensures that the sequence $(x_{k})\subset \{x\in
I\!\!R^{n},\Gamma (x)\leq \Gamma (x_{0})\}$ that is bounded since $\Gamma $
is coercive.\newline
(iii) If $x_{1}=x_{0}=0$, we then stop the algorithm producing the solution $%
x^{\ast }=0$. Otherwise, it follows from (\ref{Eq:2.7})
\begin{equation}
\Gamma (x_{0})-\Gamma (x_{1})\geq \gamma \Vert x_{1}\Vert _{2}>0,
\end{equation}%
so $x_{k}\not=0$ for all $k\geq 1$. Since $(\Gamma (x_{k}))$ is convergent,
substituting $y_{k}=\frac{x_{k}}{\Vert x_{k}\Vert _{2}}$, leads to  }

{\small
\begin{equation}
\lim_{k\rightarrow +\infty}\Vert (I-P_Q)(Ax_k)-(I-P_Q)(Ax_{k+1})\Vert^2=0 \ %
\mbox{and} \ \lim_{k\rightarrow +\infty}\Vert x_{k+1}\Vert_2-\frac{\langle
x_k,x_{k+1}\rangle}{\Vert x_k\Vert_2}=0.  \label{Eq:2.8}
\end{equation}
}

{\small Now, let $(x_{k_{\nu }})$ be a subsequence of $(x_{k})$ converging
to $x^{\ast }\not=0$, so the optimality condition at the $k_{\nu }$ the step
of Algorithm (\ref{Alg:2.3}) reads
\begin{equation}
-\left( A^{t}(I-P_{Q})Ax_{k_{\nu }}-\gamma \frac{x_{k_{\nu }-1}}{\Vert
x_{k_{\nu }-1}\Vert _{2}}\right) \in \gamma \partial \Vert x_{k_{\nu }}\Vert
_{1}+N_{C}(x_{k_{\nu }}).  \label{Eq:2.9}
\end{equation}%
Since $\lim_{\nu \rightarrow +\infty }x_{k_{\nu }}=x^{\ast }$, the operator $%
A^{t}(I-P_{Q})A$ is Lipschitz continuous, the sequence $(x_{k})$ is assumed
to be asymptotically regular and $x^{\ast }$ is away from 0, we have
\begin{eqnarray}
\lim_{\nu \rightarrow +\infty }\left( A^{t}(I-P_{Q})Ax_{k_{\nu }}-\gamma
\frac{x_{k_{\nu }-1}}{\Vert x_{k_{\nu }-1}\Vert _{2}}\right) &=&\lim_{\nu
\rightarrow +\infty }\left( A^{t}(I-P_{Q})Ax_{k_{\nu }}-\gamma \frac{%
x_{k_{\nu }}}{\Vert x_{k_{\nu }}\Vert _{2}}\right.  \notag \\
&+&\left. \gamma \left( \frac{x_{k_{\nu }}}{\Vert x_{k_{\nu }}\Vert _{2}}-%
\frac{x_{k_{\nu }-1}}{\Vert x_{k_{\nu }-1}\Vert _{2}}\right) \right)  \notag
\\
&=&A^{t}(I-P_{Q})Ax^{\ast }-\gamma \frac{x^{\ast }}{\Vert x^{\ast }\Vert _{2}%
},
\end{eqnarray}%
and by passing to the limit as $\nu \rightarrow +\infty $ in (\ref{Eq:2.9})
and by taking into account the fact that $\partial (\Vert \cdot \Vert
_{1}+i_{C})$ is a maximal monotone operator, which assures that its graph is
closed, we obtain at the limit  }

{\small
\begin{equation}
-\left( A^t(I-P_Q)Ax^*-\gamma\frac{x^*}{\Vert x^*\Vert_2} \right)\in
\gamma\partial \Vert x^*\Vert_1+N_C(x^*),
\end{equation}
}

{\small in other words $x^*$ is a stationary point.  }
\end{proof}

{\small \ The asymptotical regularity assumption is satisfied in the
particular case where $Q$ is a singleton considered in \cite{ylhx}. In what
follows, we will prove that it is also the case in the interesting setting
of closed convex cones which usually arises, for example, in statistical
applications and also in image recovery where subspaces are often used.
Likewise, when the projection has the nice property to be homogeneous with
respect to the set $Q$, which is the case, for instance, for balls,
rectangles,... when the points to project are outside.  }

\begin{proposition}
{\small The iteration sequence is asymptotically regular in the following
three cases:\newline
i) $Q=\{b\}$. \newline
ii) $Q$\ is a closed convex cone and when Q is a subspace. \newline
iii) The projection is a non-negative homogeneous function with respect to
the set $Q$, namely%
\begin{equation}
\forall \alpha >0\ \forall x\in I\!\!R^{n}\ \mbox{one has}\ P_{\alpha
Q}(x)=\alpha P_{Q}(x).
\end{equation}
}
\end{proposition}

\begin{proof}
{\small i) Indeed, in this case relation (\ref{Eq:2.7}) reduces to
\begin{equation}
\Gamma (x_k)-\Gamma (x_{k+1})\geq \frac{1}{2}\Vert
Ax_k-Ax_{k+1}\Vert^2+\gamma ( \Vert x_{k+1}\Vert_2-\Vert x_{k}\Vert_2-
\langle y_k, x_{k+1}-x_k\rangle)\geq 0,
\end{equation}
which is exactly the relation that gives the asymptotical regularity in \cite%
{ylhx}. Following the same lines of the proof of \cite{ylhx}-Proposition
3.1-(b), we obtain the desired result which is similar to the end of the
proof of iii) below.\newline
ii) In this setting the projection is a non-negative homogeneous function,
i.e.,
\begin{equation}
\forall \alpha\geq 0 \ \forall x\in I\!\!R^n \ \mbox{one has} \ P_Q(\alpha
x)=\alpha P_Q(x).
\end{equation}
At this stage, observe that this property holds true also for subspaces
since the projection is linear in this case and the proof will be the same.
Now, remember that $I-P_Q=P_{Q^*}$, where $Q^*:=\{y\in I\!\!R^n, \langle
y,x\rangle\leq 0 \ \forall x\in Q\}$ and set $c_k=\frac{\langle
x_k,x_{k+1}\rangle}{\Vert x_k\Vert^2}$ and $\varepsilon_k=x_{k+1}-c_kx_k$.
It suffices then to prove that $\lim_{k\rightarrow +\infty}\varepsilon_k=0$
and $\lim_{k\rightarrow +\infty}c_k=1$. A simple computation shows that
\begin{equation}
\Vert \varepsilon_k\Vert^2_2=\Vert x_{k+1}\Vert^2_2-\frac{\langle
x_k,x_{k+1}\rangle^2}{\Vert x_k\Vert^2}\rightarrow 0 \ \mbox{as} \
k\rightarrow +\infty,
\end{equation}
by virtue of the second limit in (\ref{Eq:2.8}). On the other hand using the
first limit in (\ref{Eq:2.8}), we can write
\begin{eqnarray}
0=\lim_{k\rightarrow +\infty }\Vert P_{Q^{\ast
}}(Ax_{k})-P_{Q^{\ast}}(Ax_{k+1})\Vert &=&\lim_{k\rightarrow +\infty }\Vert
P_{Q^{\ast }}(Ax_{k})-P_{Q^{\ast }}(A(c_{k}x_{k}+\varepsilon _{k}))\Vert
\notag \\
&=&\lim_{k\rightarrow +\infty }\Vert P_{Q^{\ast }}(Ax_{k})-P_{Q^{\ast
}}(A(c_{k}x_{k}))\Vert  \notag \\
&=&\lim_{k\rightarrow +\infty }\Vert P_{Q^{\ast }}(Ax_{k})-P_{Q^{\ast
}}(c_{k}A(x_{k}))\Vert  \notag \\
&=&\lim_{k\rightarrow +\infty }|c_{k}-1|\Vert P_{Q^{\ast }}(Ax_{k})\Vert ,
\end{eqnarray}
where we used the homogeneity of the projection and the fact that $c_k>0$.
The latter follows from the fact that $x_{k+1}$ is a minimizer in \ref%
{Alg:2.3}. More precisely, we can write
\begin{equation}
\frac{1}{2}\Vert P_{Q^*}(Ax_{k+1})\Vert^2_2+\gamma\Vert
x_{k+1}\Vert_1-\langle x_{k+1},\gamma \frac{x_k}{\Vert x_k\Vert_2}%
\rangle\leq \frac{1}{2}\Vert P_{Q^*}(A(0))\Vert^2_2+\gamma\Vert
0\Vert_1-\langle 0,\gamma \frac{x_k}{\Vert x_k\Vert_2}\rangle=0.
\end{equation}
From which we obtain that $c_k> 0$. Now, if $\lim_{k\rightarrow
+\infty}(c_k-1)\not=0$, then there exists a subsequence $( x_{k_\nu})$ such
that $\lim_{\nu\rightarrow +\infty}P_{Q^*}(Ax_{k_\nu})=0$. So, we have
\begin{equation}
\lim_{\nu\rightarrow +\infty}\Gamma(x_{k_\nu})\geq \lim_{\nu\rightarrow
+\infty}\frac{1}{2}\Vert P_{Q^*}(Ax_{k_\nu})\Vert^2=0=\Gamma(x_0),
\end{equation}
which contradicts the fact that
\begin{equation}
\Gamma(x_{k_\nu})\leq \Gamma(x_1)<\Gamma(x_0) \ \forall k_\nu\geq 1.
\end{equation}
Consequently, $\lim_{k\rightarrow +\infty}c_k=1$ and thus $%
\lim_{k\rightarrow +\infty}\Vert x_{k+1}-x_k\Vert=0$ which completes the
proof.\newline
iii) To begin with, a simple calculation shows that $P_{\alpha Q}(x)=\alpha
P_Q(\frac{1}{\alpha}x)$, see for example \cite[Lemma 2.1]{cgls17} with $U=I$
and $A=0$. Hence, we have $P_Q(c_k(Ax_k))=c_kP_{\frac{1}{c_k}Q}(Ax_k)$ and
thus
\begin{equation}
P_Q(c_k(Ax_k))=c_kP_{\frac{1}{c_k}Q}(Ax_k)= c_k\frac{1}{c_k}%
P_Q(Ax_k)=P_Q(Ax_k),
\end{equation}
by virtue of the homogeneous property of the projection and the fact that $%
c_k>0$. With this and the first limit in (\ref{Eq:2.8}) in hand, we can
successively write
\begin{eqnarray}
\lim_{k\rightarrow +\infty }\Vert (I-P_{Q})(Ax_{k})-(I-P_{Q})(Ax_{k+1})\Vert
&=&\lim_{k\rightarrow +\infty }\Vert
(I-P_{Q})(Ax_{k})-(I-P_{Q})(A(c_{k}x_{k}+\varepsilon _{k}))\Vert  \notag \\
&=&\lim_{k\rightarrow +\infty }\Vert
(I-P_{Q})(Ax_{k})-(I-P_{Q})(A(c_{k}x_{k}))\Vert  \notag \\
&=&\lim_{k\rightarrow +\infty }\Vert
(I-P_{Q})(Ax_{k})-(I-P_{Q})(c_{k}A(x_{k}))\Vert  \notag \\
&=&\lim_{k\rightarrow +\infty }|c_{k}-1|\Vert Ax_{k}\Vert =0.
\end{eqnarray}
Now, if $\lim_{k\rightarrow +\infty}(c_k-1)\not=0$, then there exists a
subsequence $( x_{k_\nu})$ of $(x_k)$ such that $\lim_{\nu\rightarrow
+\infty}Ax_{k_\nu}=0$. So, we have
\begin{equation}
\lim_{\nu\rightarrow +\infty}\Gamma(x_{k_\nu})\geq \lim_{\nu\rightarrow
+\infty}\frac{1}{2}\Vert (I-P_Q)(Ax_{k_\nu})\Vert^2=\frac{1}{2}\Vert
P_Q(0)\Vert^2=\Gamma(x_0),
\end{equation}
which contradicts the fact that
\begin{equation}
\Gamma(x_{k_\nu})\leq \Gamma(x_1)<\Gamma(x_0) \ \forall k_\nu\geq 1.
\end{equation}
Consequently, $\lim_{k\rightarrow +\infty}c_k=1$ and again the sequence $%
(x_k)$ is asymptotically regular.  }
\end{proof}

\begin{remark}
{\small Each DCA iteration requires solving a $l_1$-regularized split
feasibility subproblem of the form
\begin{equation}
\min_{x\in C}(\frac{1}{2}\Vert (I-P_Q)Ax\Vert^2_2+\langle x,v\rangle
+\gamma\Vert x\Vert_1),
\end{equation}
where $v\in I\!\!R^n$ is a constant vector. This problem can be done by the
two split proximal algorithms (coupling forward-backward and the
Douglas-Rachford algorithms) proposed in \cite{ylhx}, \cite{Moudafi16} and
also by the alternating direction method of multipliers (ADMM) following the
analysis developed in \cite{Tibshirani96} for the special case were $Q$ is a
singleton. The details will be given in the appendix.}
\end{remark}

\subsection{\protect\small Forward-backward splitting algorithm}

{\small To begin with, recall that the proximal mapping (or the Moreau
envelope) of a proper, convex and lower semicontinuous function $\varphi $
of parameter $\lambda >0$ is defined by
\begin{equation}
prox_{\lambda \varphi }(x):=arg\min_{v\in I\!\!R^{n}}\left\{ \varphi (v)+%
\frac{1}{2\lambda }\Vert v-x\Vert ^{2}\right\} ,\ x\in I\!\!R^{n},
\label{eq:Prox_def}
\end{equation}%
and that it has closed-form expression in some important cases. For example,
if $\varphi =\Vert \cdot \Vert _{1}$, then for $x\in I\!\!R^{n}$  }

{\small
\begin{equation}
prox_{\lambda \Vert \cdot \Vert _{1}}(x)=(prox_{\lambda |\cdot
|}(x_{1}),prox_{\lambda |\cdot |}(x_{n})),
\end{equation}%
where $prox_{\lambda |\cdot |}(x_{k})=sgn(x_{k})\max_{k=1,2,\cdot \cdot
\cdot n}\{|x_{k}|-\lambda ,0\}$.\newline
If $\varphi =i_{C}$, we have
\begin{equation}
prox_{\gamma \varphi }(x)=Proj_{C}(x):=arg\min_{z\in C}\Vert x-z\Vert .
\end{equation}
}

{\small For sake of simplicity and clarity, we set in what follows $%
C=I\!\!R^n$. Observe that when $\gamma>0$, the minimization problem (\ref{Eq:1.9}) can be
written as
\begin{equation}
\min_{x\in I\!\!R^n}\frac{1}{2\gamma}\Vert (I-P_Q)Ax\Vert^2_2+\Vert
x\Vert_1-\Vert x\Vert_2.
\end{equation}
It is worth mentioning that when $C\not=I\!\!R^n$, this requires to compute
the proximal operator of a sum, namely $prox_{i_C+\gamma_k(\Vert\cdot%
\Vert_1-\Vert\cdot\Vert_2)}$ which may be performed with Douglas-Rachford
iterations in the spirit of the analysis developed in \cite{cpp09} and \cite%
{Moudafi16}. \smallskip\newline
A closed-form solution of $prox_{\Vert x\Vert_1-\Vert x\Vert_2}$ was
proposed in \cite{ly}, in particular we have the following lemma.  }

\begin{lemma}
{\small Given $y\in I\!\!R^n$, $\lambda>0$ and setting $r(x)= \Vert
\cdot\Vert_1-\Vert \cdot\Vert_2$, we have\newline
(i) When $\lambda<\Vert y\Vert_\infty$, then
\begin{equation}
prox_{\lambda r}(y)=\frac{\lambda+\Vert prox_{\lambda \Vert
\cdot\Vert_1}y\Vert_2}{\Vert prox_{\lambda \Vert \cdot\Vert_1}y\Vert_2}%
prox_{\lambda\Vert \cdot\Vert_1}y.
\end{equation}
(ii) When $\lambda=\Vert y\Vert_\infty$, then $x^*\in prox_{\lambda r}(y)$
if and only if it satisfies $x^*_i=0$ if $\vert y_i\vert<\lambda, \Vert
x^*\Vert_2=\lambda$ and $x^*_iy_i\geq 0$ for all $i$.\newline
(iii) When $\lambda>\Vert y\Vert_\infty$, then $x^*\in prox_{\lambda r}(y)$
if and only if it is a 1-spare vector satisfying $x^*_i=0$ if $\vert
y_i\vert<\Vert y\Vert_\infty, \Vert x^*\Vert_2=\Vert y\Vert_\infty$ and $%
x^*_iy_i\geq 0$ for all $i$.\  }
\end{lemma}

{\small By setting $l(x)=\frac{1}{2\gamma }\Vert (I-P_{Q})Ax\Vert _{2}^{2}$,
the forward-backward splitting algorithm can be expressed as follows
\begin{equation}
x_{k+1}\in prox_{\lambda r}(x_{k}-\lambda \nabla l(x_{k})).  \label{eq:FB}
\end{equation}%
Since the two assumptions of \cite[Theorem 3]{ly} are satisfied, namely the
coerciveness of the objective function and differentiability of the function
$l$ with Lipschitz-continuity of its gradient, a direct application of this
Theorem leads to the following convergence result:  }

\begin{proposition}
{\small If $\lambda<\frac{\gamma}{\Vert A\Vert^2}$, then the objective
values are decreasing and there exists a subsequence of $(x_k)$ that
converges to a stationary point. Furthermore, any limit point of $(x_k)$ is
a stationary point.  }
\end{proposition}

\subsection{\protect\small Mine-Fukushima Algorithm}

{\small At this stage, we would like to mention that in the case where $C$
is strictly convex and that we can generate from an initial point $x_{0}$ a
sequence $x_{k}$ such that $x_{k}\not=0$ for all $k\in I\!\!N$, then the
Algorithm introduced by Mine-Fukushima in \cite{mf81} is applicable. Indeed,
problem (\ref{Eq:1.9}) can be written as
\begin{equation}
\min_{x\in I\!\!R^{n}}(\phi (x):=f(x)+g(x)),  \label{Eq:2.14}
\end{equation}%
with $f(x)=\frac{1}{2}\Vert (I-P_{Q})Ax\Vert _{2}^{2}-\gamma \Vert x\Vert
_{2}$ and $g(x)=\gamma \Vert x\Vert _{1}+i_{C}(x)$. Observe that in this
case, we have for $x\not=0$, that $\nabla f(x)=A^{t}(I-P_{Q})Ax-\gamma \frac{%
x}{\Vert x\Vert _{2}}$ and $\partial g(x)=\partial \Vert x\Vert
_{1}+N_{C}(x) $.  }

{\small So \cite[Algorithm 2.1]{mf81} take the following from:\newline
}

{\small \textbf{Algorithm (Mine-Fukushima):}\label{Alg:MF}\newline
Step 1. Let $x_{0}$ be any initial point. Set $k=0$, and go to step 2.%
\newline
Step 2. If $-\nabla f(x_{k})\in \partial g(x_{k})$, stop. Otherwise, go to
Step 3.\newline
Step 3. Find a minimum $\tilde{x}_{k}$ of
\begin{equation}
\min_{x\in C}\left( \left\langle x,A^{t}(I-P_{Q})Ax_{k}-\gamma \frac{x_{k}}{%
\Vert x_{k}\Vert _{2}}\right\rangle +\gamma \Vert x\Vert _{1}\right) ,
\label{Eq:2.15}
\end{equation}%
and go to Step 4.\newline
Step 4. Find
\begin{equation}
x_{k+1}=\lambda _{k}\tilde{x}_{k}+(1-\lambda _{k})x_{k},
\end{equation}%
such that $\lambda _{k}\geq 0$ and
\begin{equation}
\phi (x_{k+1})\leq \phi (\lambda \tilde{x}_{k}+(1-\lambda )x_{k})\
\mbox{for
all }\ \lambda \geq 0.
\end{equation}%
Set $k=k+1$, and go to Step 2. \smallskip \newline
Observe that solving (\ref{Eq:2.15}) in Step 3 is equivalent to finding $%
\tilde{x}_{k}$ such that $-\nabla f(x_{k})\in \partial g(\tilde{x}_{k})$.
\smallskip \newline
Since $\phi $ is coercive in our case, a direct application of \cite[Theorem
3.]{mf81} yields the following result.  }

\begin{proposition}
{\small The sequence $(x_{k})$ generated by the Mine-Fukushima Algorithm
contains a subsequence which converges to a critical point $x^{\ast }$ of (%
\ref{Eq:2.14}), namely
\begin{equation}
-A^{t}(I-P_{Q})Ax^{\ast }-\gamma \frac{x^{\ast }}{\Vert x^{\ast }\Vert _{2}}%
\in \partial \Vert x^{\ast }\Vert _{1}+N_{C}(x^{\ast }).
\end{equation}
}
\end{proposition}

\begin{remark}
{\small The assumption of strict convexity on the convex set $C$ can be
removed by applying the following process: for some $\mu >0$ consider the
following decomposition of the objective function $\phi $: $\phi =\tilde{f}+%
\tilde{g}$ with $\tilde{f}(x)=f(x)-\mu \frac{\Vert x\Vert _{2}^{2}}{2}$ and $%
g$ by $\tilde{g}(x)=g(x)+\mu \frac{\Vert x\Vert _{2}^{2}}{2}$. Relation (\ref%
{Eq:2.15}) becomes
\begin{equation}
\min_{x\in C}\left( \left\langle x,A^{t}(I-P_{Q})Ax_{k}+\mu x_{k}-\gamma
\frac{x_{k}}{\Vert x_{k}\Vert _{2}}\right\rangle +\gamma \Vert x\Vert
_{1}+\mu \frac{\Vert x\Vert _{2}^{2}}{2}\right) .
\end{equation}
}
\end{remark}

\section{\protect\small Numerical experiments}

{\small \label{Sec:Num} In this section, we present two numerical examples
demonstrating the performances of our proposed schemes. In both experiments
we wish to solve the linear system of equations: $Ax=b$ with $A\in
I\!\!R^{120\times 512}$. In the first example we generate 50 random problems
from a normal distribution with mean zero and variance one. For the second
experiment we choose a problem in the field of compressed sensing, which
consists of recovering a sparse signal $x\in I\!\!R^{512}$ with 50 non zero
elements from $120$ measurements. In this case we also include noise, that
is, we wish to solve $Ax=b+\varepsilon $, where $\varepsilon $ is the noise
with bounded variance $10^{-4}$.  }

{\small For the comparison of our proposed schemes we decided also to
include Byrne CQ algorithm \cite{Byrne02, Byrne04} and Qu and Xiu \cite{qx05}
modified CQ algorithm. Byrne CQ algorithm is designed to solve $Ax=b$ and
hence we choose $C=I\!\!R_{+}^{n}$ and $Q=\{b\}$. The CQ iterative step
reads as follows
\begin{equation}
x_{k+1}=P_{C}(x_{k}-\widehat{\gamma} A^{t}(I-P_{Q})Ax_{k})
\end{equation}%
and for the specific choice of $C$\ and $Q$ it translates to%
\begin{equation}
x_{k+1}=\left(x_{k}-\widehat{\gamma} A^{t}(Ax_{k}-b)\right)_+  \label{eq:CQ}
\end{equation}%
and it is denoted in our plots (Figures \ref{fig:firstex} and \ref%
{fig:secondex}) as \textbf{CQ}.\newline
\medskip Qu and Xiu \cite{qx05} modified CQ algorithm (see also Tang et al.
\cite{tl16}) uses subgradient (elements of the subdifferential set) projection onto super-sets $C\subseteq C_{k}$
and $Q\subseteq Q_{k}$ instead of the orthogonal projections onto $C$ and $Q$%
. The algorithm also make use of adaptive step-size $\alpha _{k}$\ instead
of fixed $\widehat{\gamma}$ as in the CQ algorithm. The algorithm is as follows.%
\newline
}

{\small \textbf{Algorithm (Modified CQ):}\label{Alg:MCQ}\newline
Step 1. Given constants $l,\mu \in (0,1)$ and choose $x_{0}$$%
\in I\!\!R^{n}$. Set $k=0$, and go to step 2.\newline
Step 2. Given the current iterate $x_{k}$, let
\begin{equation}
\overline{x_{k}}=P_{C_{k}}(x_{k}-\alpha _{k}A^{t}(I-P_{Q})Ax_{k})
\end{equation}%
where $\alpha _{k}l^{m_{k}}$ and $m_{k}$ is the smallest nonnegative integer
$m$ such that
\begin{equation}
\Vert A^{t}(I-P_{Q})Ax_{k}-A^{t}(I-P_{Q})A\overline{x_{k}}\Vert \leq \mu
\frac{\Vert x_{k}-\overline{x_{k}}\Vert }{\alpha _{k}}.
\end{equation}%
And the next iterate is calculated via
\begin{equation}
x_{k+1}=P_{C_{k}}(x_{k}-\alpha _{k}A^{t}(I-P_{Q})A\overline{x_{k}}).
\end{equation}%
Set $k=k+1$, and go to Step 2. \smallskip \newline
}

{\small While for the CQ algorithm we wish to solve $Ax=b$, for the modified
CQ algorithm we wish to the consider $Ax=b$\ with $l1$\ regularization, this
is known as the LASSO problem \cite{Tibshirani96} (strongly related to the
Basis Pursuit denosing problem \cite{chens1998})
\begin{equation}
\min_{x\in C}\frac{1}{2}\Vert Ax-b\Vert _{2}^{2}\quad \mbox{subject to}\
\Vert x\Vert _{1}\leq t
\end{equation}%
where $t>0$ is a given constant. So in this case we choose $C=\{x\mid \Vert
x\Vert _{1}\leq t\}$\ and $Q=\{b\}$. We define the convex function $%
c(x)=\Vert x\Vert _{1}-t$\ and denote the level set $C_{k}$\ by,
\begin{equation}
C_{k}=\{x\mid c(x^{k})+\langle \xi _{k},x-x^{k}\rangle \leq 0\},
\end{equation}%
where $\xi _{k}\in \partial c(x_{k})$ is an element (subgradient) from the subdifferential of $c $ at $x_{k}$. The orthogonal projection
onto $C_{k}$\ can be calculated by the following,
\begin{equation}
P_{C_{k}}(y)=%
\begin{cases}
y, & \text{if }c(x_{k})+\langle \xi _{k},y-x_{k}\rangle \leq 0, \\
y-\frac{c(x_{k})+\langle \xi _{k},y-x_{k}\rangle }{\Vert \xi _{k}\Vert ^{2}}%
\xi _{k}, & \text{otherwise}.%
\end{cases}%
\end{equation}%
Following the definition of the subdifferential set $\partial c(x_{k})$ (\ref{Eq:Subdiff}), we choose subgradient $\xi _{k}\in \partial c(x_{k})$ as
\begin{equation}
(\xi _{k})_i=
\begin{cases}
1, & (x_{k})_i>0, \\[0pt]
0, & (x_{k})_i\neq 0, \\
-1, & (x_{k})_i<0.%
\end{cases}%
\end{equation}%
This algorithm, Algorithm \ref{Alg:MCQ},\ is denoted in our plots (Figures %
\ref{fig:firstex} and \ref{fig:secondex}) as \textbf{Mod CQ}\textbf{\ }($%
\mathbf{l1}$\textbf{-con.}). }

{\small Our schemes, DC (difference of convex) algorithm (DCA)-iterative
step (\ref{Alg:2.3}), the forward-backward (FB) algorithm-iterative step (%
\ref{eq:FB}) and the Mine and Fukushima algorithm-Algorithm \ref{Alg:MF} are
denoted in our plots (Figures \ref{fig:firstex} and \ref{fig:secondex}) as
\textbf{DCA} ($\mathbf{l_{1}}$-$\mathbf{l_{2}}$), \textbf{FB} and \textbf{%
Mine and Fukushima}, respectively. The stopping criterion for all schemes is
either 1000 iterations or until $\Vert x_{k+1}-x_{k}\Vert <10^{-5}$ is
reached. In the experiments we choose arbitrary the regularization parameter $\gamma$ to be $0.6$. We noticed that this choice produce good results, and this also affects the sensitivity of the solution.
All computations were performed using MATLAB R2015a on an Intel
Core i5-4200U 2.3GHz running 64-bit Windows. }

{\small Next the two numerical illustrations are presented. In Figures \ref%
{fig:firstex} and \ref{fig:secondex} we present the performances of our
schemes as well as the CQ and the Modified CQ algorithms for random data and
sparse signal recovering, respectively. As explained above the algorithms
are designed to solve $Ax=b$\ with and without different types of
regularizations. In Figure \ref{fig:firstex}, we present the results for the
50 random generated problems, in each plot the different colors represent
the quintiles with respect to each of the 50 problems and the red graph is
the experiments median. It can be seen that most methods differ in their
"warmup" stages, that is, in the first number of iterations, and all
converge "quite fast", just within a few iterations. We see that the
"warmup" stage in the DCA is the most significant and visible, we suspect
that this is probably due to the need to solve subproblems during each iteration. This would probably play an essential role as a computational aspect for large scale problems. In Figure \ref{fig:secondex}, we test the
5 schemes performances for recovering a 50-sparse signal $x\in I\!\!R^{512}$
from $120$ measurements. Here, when only the resulting recovered signal is
presented, it can be seen that the DCA, FB and Mine and Fukushima algorithms
recover the exact signal while both the CQ and the modified CQ
algorithms contain errors, and as expected the modified CQ algorithm
generate a slightly better signal, probably due to the $\mathbf{l1}$%
-regularization. We would like to emphasize that the main goal of this work
is to introduce and survey some approaches for solving $Ax=b$\ with
different variants of regularizations, we dont wish to further investigate
and analyze the computational performances of the proposed schemes, and
hence wish to leave our above explanations as compact as possible. An
interesting direction for future study is indeed a computational comparison
between different types of regualrizations. We believe that deep insights in
this case can be derived, only when large problems are considered, since
then the subproblems solved per each iteration in the related algorithms,
might play an essential role with respect to the computational efforts and
convergence rate, and moreover, this could emphasize and suggest the
applicability and advantages of the different methods and in particular the
usage of one regularization over another. }

{\small
\begin{figure}[htp]
\begin{center}
{\small \includegraphics[width=12cm,height=8cm]{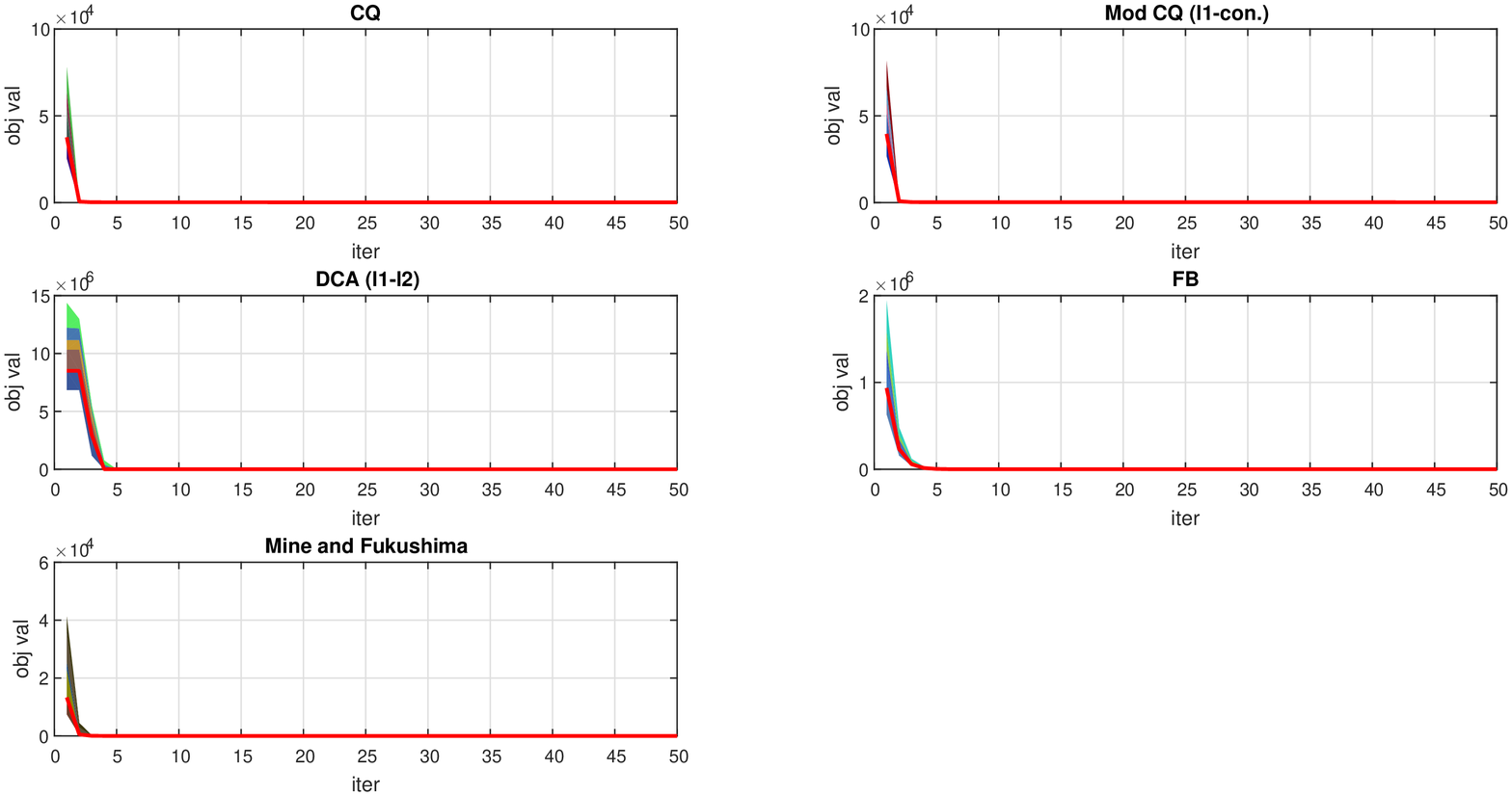}  }
\end{center}
\par
{\small \  }
\caption{Testing our proposed algorithms for 50 random problems $Ax=b$,
where $A\in I\!\!R^{120\times 512}$.}
\label{fig:firstex}
\end{figure}
}

{\small
\begin{figure}[htp]
\begin{center}
{\small \includegraphics[width=12cm,height=8cm]{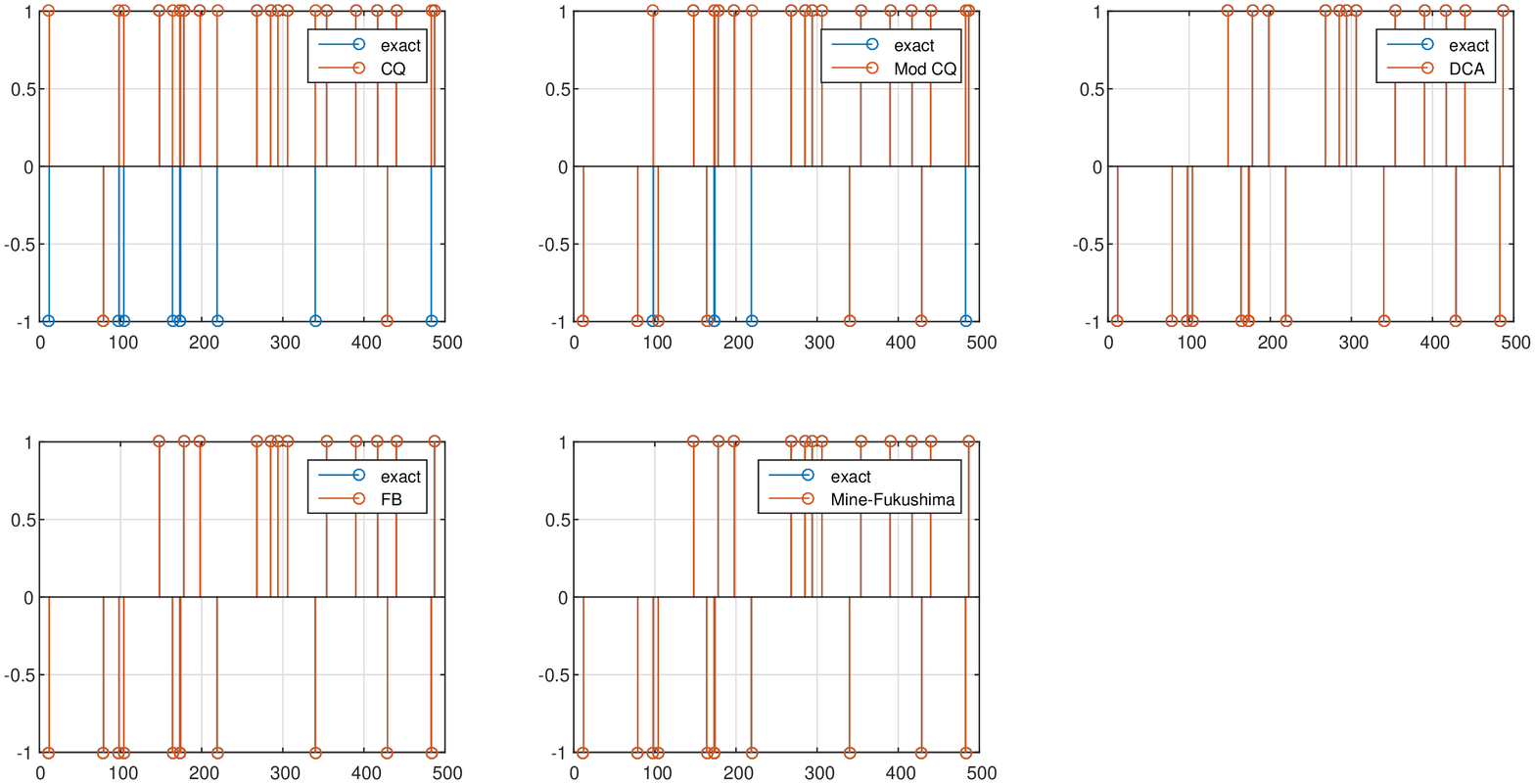}  }
\end{center}
\par
{\small \  }
\caption{Testing our proposed algorithms for recovering a 50-sparse signal $%
x\in I\!\!R^{512}$ from $120$ measurements.}
\label{fig:secondex}
\end{figure}
}

{\small \newpage }

\section*{\protect\small Concluding remarks}

{\small In this paper we investigate split feasibility problems under a
nonconvex Lipschitz continuous metric instead of conventional methods such
as $l_{1}$ or $l_{1}-l_{2}$ minimization, for example in \cite{aasnx13}. We
present and analyze the convergence to a stationary point of an iterative
minimization method based on DCA (difference of convex algorithm), see for
example \cite{Natarajan95}). Furthermore, relying on a proximal operator for
$l_{1}-l_{2}$ as well as on an algorithm proposed by Mine and Fukushima for
minimizing the sum of a convex function and a differentiable one, two
additional algorithms are presented and their convergence properties are
discussed. }

{\small Since each iteration of the DCA requires to solve an inner $l_{1}$%
-regularized split feasibility subproblem, we present some algorithms
designed for that purpose in the appendix. Observe that the DCA presented here can
be extended to split fixed-point problems governed by firmly
quasi-nonexpansive mappings. We would also like to emphasize that much
attention has been paid not only to the sparsity of solutions but also to
the structure of this sparsity, which may be relevant in some problems and
which provides another avenue for inserting prior knowledge into the
problem. We would like to mention that an interesting regularizer is the
OSCAR one which has the following form
\begin{equation}
r_{OSCAR}(x)=\gamma _{1}\Vert x\Vert _{1}+\gamma _{2}\sum_{i<j}\max
\{|x_{i}|,|x_{j}|\}.
\end{equation}%
Due to $l_{1}$ term and the pairwise $l_{\infty }$ penalty, the components
are encouraged to standard spare and pairwise similar magnitude and has been
extensively applied in various feature grouping tasks and outperforms other
models. We refer to the interesting paper \cite{zf14} where the OSCAR
regularizer is used via its proximity mapping, a work that deserves to be
more developed. }

\section*{\protect\small Acknowledgment}

{\small We wish to thank the anonymous referee for the thorough analysis and
review, all the comments and suggestions helped tremendously in improving
the quality of this paper and made it suitable for publication. }

{\small The first author wish to thank his team "Image \& Mod\`{e}le" and
the Computer Science \& System Laboratory (L.I.S) of Aix-Marseille
University. The second author work is supported by the EU FP7 IRSES program
STREVCOMS, grant no. PIRSES-GA-2013-612669.\medskip  }

\section{\protect\small Appendix}

{\small \label{Sec:App} Each DCA iteration requires solving a $l_1$%
-regularized split feasibility subproblem of the form  }

{\small
\begin{equation}
\min_{x\in C}\frac{1}{2}\Vert (I-P_Q)Ax\Vert^2_2+\langle x,v\rangle
+\gamma\Vert x\Vert_1,  \label{Eq:3.1}
\end{equation}
where $v\in I\!\!R^n$ is a constant vector. This problem can be done, for
example, by the two split proximal algorithms (coupling forward-backward and
the Douglas-Rachford algorithms).  }

\subsection{\protect\small Insertion of a forward-backward step in the
Douglas-Rachford algorithm}

{\small To apply the Douglas-Rachford algorithm when $g_{1}=\gamma \Vert
\cdot \Vert _{1}$ and $g_{2}=\frac{1}{2}\Vert (I-P_{Q})A(\cdot )\Vert
_{2}^{2}+\langle \cdot,v\rangle +i_{C}$, we need to determine their proximal
mappings. The main difficulty lies in the computation of the second one,
namely $prox_{\kappa \frac{1}{2}\Vert (I-P_{Q})A(\cdot )\Vert
_{2}^{2}+\langle \cdot,v\rangle +i_{C}}$. As in \cite{cpp09}, we can use a
forward-backward algorithm to achieve this goal.\newline
The resulting algorithm is: \medskip \newline
\noindent \textbf{Algorithm:}\newline
Step 1. Set $\underline{\gamma }\in ]0,2\kappa ^{-1}\Vert A\Vert ^{-1}],%
\underline{\lambda }\in ]0,1]$ and $\kappa \in ]0,+\infty \lbrack $. \newline
\qquad Choose $(\tau _{k})_{k\in I\!\!N}$ satisfying $\forall k\in I\!\!N,\
\tau _{k}\in ]0,2[$ and $\sum_{k=0}^{\infty }\tau _{k}(2-\tau _{k})=+\infty $%
\newline
Step 2. Set $k=0,y_{0}=y_{-1/2}\in C$\newline
Step 3. Set $x_{k,0}=y_{k-1/2}$\newline
Step 4. For $i=0,\ldots ,N_{k}-1$\newline
\quad a) Choose $\gamma _{k,n}\in \lbrack \underline{\gamma },2\kappa
^{-1}\Vert A\Vert ^{-1}[$ and $\lambda _{k,i}\in \lbrack \underline{\lambda }%
,1]$.;\newline
\quad b) Compute
\begin{equation}
x_{k,i+1}=x_{k,i}+\lambda _{k,i}\left( P_{C}(\frac{x_{k,i}-\gamma
_{k,i}(\kappa (A^{t}(I-P_{Q})Ax_{k,i}+v_{i})-y_{k})}{1+\gamma _{k,i}}\big)%
-x_{k,i}\right) .
\end{equation}%
Step 5. Set $y_{k+1/2}=x_{k,N_{k}}$\newline
Step 6. Set $y_{k+1}=y_{k}+\tau _{k}(prox_{\kappa {\Vert \cdot\Vert}_{1}}(2y_{k+1/2}-y_{k})-y_{k+1/2})$.\smallskip \newline
Step 7. Increment $k\leftarrow k+1$ and go to Step 3.  }

{\small By a judicious choose of of $N_k$, the convergence of the sequence $%
(y_k)$ to $y$ such that
\begin{equation}
prox_{\kappa (\frac{1}{2}\Vert (I-P_Q)A(\cdot)\Vert^2_2+\langle
\cdot,v\rangle)+i_C}(y)
\end{equation}
solves problem (\ref{Eq:3.1}), follows directly by applying \cite[%
Proposition 4.1]{cpp09}.  }

\subsection{\protect\small Insertion of a Douglas-Rachford step in the
forward-backward algorithm}

{\small We consider $f_{1}=\kappa \Vert \cdot \Vert _{1}+i_{C}$ et $f_{2}=%
\frac{1}{2}\Vert (I-P_{Q})A(\cdot )\Vert _{2}^{2}+\langle \cdot,v\rangle $.
Since $f_{2}$ has a $\Vert A\Vert ^{2}$-Lipschitz gradient, we can apply the
forward-backward algorithm. This requires however to compute $%
prox_{i_{C}+\gamma _{k}\Vert \cdot \Vert }$ which can be performed with
Douglas-Rachford iterations. The resulting algorithm is\smallskip \newline
\noindent \textbf{Algorithm:}\newline
Step 1. Choose $\gamma _{k}$ and $\lambda _{k}$ satisfying assumptions $%
0<\inf_{k}\gamma _{k}\leq \sup_{k}\gamma _{k}<2/\Vert A\Vert ^{2}$, $0<%
\underline{\lambda }\leq \lambda _{k}\leq 1$. \newline
Set $\underline{\tau }\in ]0,2]$.\newline
Step 2. Set $k=0,x_{0}\in C$\newline
Step 3. Set $x_{k}^{\prime }=x_{k}-\gamma _{k}(A^{t}(I-P_{Q})Ax_{n}+v)$.%
\newline
Step 4. Set $y_{k,0}=2prox_{\gamma _{k}\Vert \cdot \Vert _{1}}x_{k}^{\prime
}-x_{k}^{\prime }$.\newline
Step 5. For $i=0,\ldots ,M_{k}-1$.\newline
\quad a) Compute
\begin{equation}
y_{k,i+1/2}=P_{C}\left( \frac{y_{k,i}+x_{k}^{\prime }}{2}\right)
\end{equation}%
b) Choose $\tau _{k,i}\in \lbrack \underline{\tau },2]$.\newline
\quad c) Compute $y_{k,i+1}=y_{k,i}+\tau _{k,i}(prox_{\gamma _{k}\Vert \cdot
\Vert _{1}}(2y_{n,i+1/2}-y_{k,i})-y_{n,i+1/2})$.\newline
\quad d) If $y_{k,i+1}=y_{k,i}$, then goto Step 6.\newline
Step 6. Set $x_{k+1}=x_{k}+\lambda _{k}(y_{k,i+1/2}-x_{k})$.\newline
Step 7. Increment $k\leftarrow k+1$ and go to Step 3. \newline
}

{\small A direct application of \cite[Proposition 4.2]{cpp09} ensures the
existence of positive integers $(\overline{M}_{k})$ such that if $\forall
k\geq 0\ M_{k}\geq \overline{M}_{k}$, then the sequence $(x_{k})$ weakly
convergences to a solution of problem (\ref{Eq:3.1}).  }

\begin{remark}
{\small Other split proximal algorithms may be designed by combining the
fixed-point idea to compute the composite of a convex function with a linear
operator introduced in \cite{msxz13} and the analysis developed for
computing the proximal mapping of the sum of two convex functions developed
in \cite{cpp09} and \cite{Moudafi16}. Primal-dual algorithms considered in
\cite{Condat14} can also be used. Note that there are often several ways to
assign the functions of (\ref{Eq:3.1}) to the terms used in the generic
problem.  }
\end{remark}

\end{document}